\documentclass{amsart}

\usepackage{amsmath}
\usepackage{amssymb}
\usepackage{enumitem}
\usepackage{todonotes}

\usepackage{bbding}
\usepackage{scalerel}

\usepackage{amsthm}

\newtheorem{theorem}{Theorem}[section]
\newtheorem{lemma}[theorem]{Lemma}
\newtheorem{corollary}[theorem]{Corollary}

\theoremstyle{definition}
\newtheorem{example}{Example}
\newtheorem{definition}{Definition}

\theoremstyle{remark}
\newtheorem*{remark}{Remark}

\usepackage{hyperref}
\usepackage{graphicx}

\newcommand{\el}{\mathfrak{l}}
\newcommand{\er}{\mathfrak{r}}
\newcommand{\up}{\el}
\newcommand{\down}{\er}

\newcommand{\bifree}{\mathbin{\text{\FiveStarOutlineHeavy}}}
\newcommand{\free}{\mathbin{\scaleobj{1.2}{{\star}}}}
\newcommand{\tensor}{\mathbin{\otimes}}
\newcommand{\mon}{\mathbin{\vartriangleright}}
\newcommand{\antimon}{\mathbin{\vartriangleleft}}
\newcommand{\bimon}{\mathbin{\bowtie}}
\newcommand{\bool}{\mathbin{\scaleobj{1.2}{\diamond}}}

\newcommand{\BIGOP}[1]{\mathop{\mathchoice%
{\raise-0.22em\hbox{\huge $#1$}}%
{\raise-0.05em\hbox{\Large $#1$}}{\hbox{\large $#1$}}{#1}}}

\newcommand{\BIGboxplus}{\mathop{\mathchoice%
{\raise-0.35em\hbox{\huge $\boxplus$}}%
{\raise-0.15em\hbox{\Large $\boxplus$}}{\hbox{\large
$\boxplus$}}{\boxplus}}}

\begin{document}

\title{Bimonotone Brownian Motion}

\author{Malte Gerhold}
\date{\today}
\address{M.G., Institut f\"ur Mathematik und Informatik \\ 
Ernst Moritz Arndt Universit\"at 
Greifswald\\Walther-Rathenau-Stra\ss{}e 47 \\
17487 Greifswald \\ Germany}
\email{mgerhold@uni-greifswald.de}
\urladdr{www.math-inf.uni-greifswald.de/index.php/mitarbeiter/282-malte-gerhold}
\thanks{The author thanks Prof.\ Michael Sch{\"u}rmann for several remarks which helped to improve the readibility of this paper. He also thanks Volkmar Liebscher for some enlightning discussions. }

\begin{abstract}
We define bi-monotone independence, prove a bi-monotone central limit theorem and use it to study the distribution of bi-monotone Brownian motion, which is defined as the two-dimensional operator process with monotone and antimonotone Brownian motion as components. 
\end{abstract}

\maketitle

\section{Introduction}
\label{sec:introduction}

It is a key feature of noncommutative probability that there is not a unique notion of independence, but a number of different ones, each of which allows to a distinct theory with its own limit theorems and limit distributions, L{\'e}vy processes, de Finetti theorems, etc.  Although weaker concepts have been considered, a full theory comparable to classical probability theory is only achieved with a \emph{universal product independence}, i.e.\ independence for a family of noncommutative random variables is defined as factorization of the joint distribution with respect to a universal product construction for distributions of noncommutative random variables; this product construction replaces the tensor product of probability measures in the classical case. One possibility is to use the tensor product of states, and this leads to the notion of independence most frequently used in quantum mechanics. The most famous example of an alternative product is the free product of states which leads to the notion of \emph{freeness}, and which has been studied extensively, exhibiting many connections to random matrix theory and the theory of operator algebras, cf.\ \cite{HiPe00,VDN92}. Based on work of Speicher \cite{Spe97} and Ben Ghorbal and Sch{\"u}rmann \cite{BGS02}, Muraki showed that besides the tensor product and the free product, there are only three more universal products of states, namely the Boolean, the monotone and the antimonotone product \cite{Mur03, Mur13b}. Extending Muraki's result to a purely algebraic setting, Gerhold and Lachs classified all universal products of linear functionals on associative algebras; these are all deformations of Muraki's five, the Boolean products admits a two-parameter deformation, while the others admit a one-parameter deformation \cite{GLa14}. Results on cummulants and L{\'e}vy processes for general universal product independences can be found in \cite{Fra06,MaSc16,GLS16}.  

In 2014, Voiculescu introduced the notion of bi-freeness in a series of papers \cite{Voi14,Voi16a,Voi16b}. The main difference between bi-freeness and the universal product independences mentioned above is that Voiculescu does not consider a notion of independence for single random variables, but for pairs of random variables, or \emph{two-faced} random variables. This corresponds to considering only such algebras which are structured as a free product of subalgebras. Since universality only has to hold for algebra homomorphisms which respect the two-faced structure, the possibility arises to find universal products of states besides Muraki's five. Bi-freeness is built on the fact that there are two distinct natural realizations of the GNS-representation of a free product state, a left and a right one. The monotone product was introduced by Muraki \cite{Mur95,Mur96} and Lu \cite{Lu97}. In contrast to the free case, the monotone product is not symmetric, i.e.\ the monotone product $\mon$ and its opposite, the anti-monotone product $\varphi_1\antimon\varphi_2:=\varphi_2\mon\varphi_1$, do not coincide. There is a  ``left'' and a ``right'' product representation, which corresponds to the monotone product and the anti-monotone product, respectively.
In this paper we will introduce a new universal product for states on two-faced $*$-algebras, the \emph{bi-monotone product}, and study its associated notion of independence for two-faced random variables (Section~\ref{sec:bimon-indep}), construct a bi-monotone Brownian motion on monotone Fock space (Section~\ref{sec:bimon-brown-moti}), introduce bi-monotone partitions (Section~\ref{sec:bi-monot-part}), and use these to give a combinatorial description of the distribution of bi-monotone Brownian motion via a bi-monotone central limit theorem (Sections~\ref{sec:bi-monotone-central} and \ref{sec:distr-bi-monot}).

\section{Preliminaries and notation}
\label{sec:prel-notat}

For any natural number $n$, we denote by $[n]$ the set $\{1,\ldots, n\}$. If the set $[n]$ appears as an argument inside usual brackets, we omit the inner square bracket.  By an \emph{algebra}, we always mean a not necessarily unital, complex, associative algebra.
The free product of algebras $A_1,\ldots, A_n$ is denoted by $A_1\sqcup\cdots\sqcup A_n$.
We write
 $A=A^{(1)}\sqcup\cdots\sqcup A^{(n)}$ if $A^{(1)},\ldots, A^{(n)}$ are all subalgebras of $A$ and the canonical homomorphism $A^{(1)}\sqcup\cdots\sqcup A^{(n)}\to A$ (i.e.\ the free product of the embeddings $\iota^{(i)}\colon A^{(i)}\hookrightarrow A$) is an isomorphism. An algebra $A$ with subalgebras $A^{(1)},\ldots, A^{(n)}$ such that
 $A=A^{(1)}\sqcup\cdots\sqcup A^{(n)}$ is called an \emph{$n$-faced algebra}. Let $B_1,B_2$ be $n$-faced algebras. Then the free product $B_1\sqcup B_2$ is again an $n$-faced algebra with respect to the subalgebras $B^{(i)}:=B_1^{(i)}\sqcup B_2^{(i)}$. Recall that an \emph{augmented algebra} is a unital algebra $A$ with a character, i.e.\ a non-zero homomorphism to $\mathbb C$, whose kernel is called \emph{augmentation ideal} (cf.\ for example \cite{LoVa12}). For any algebra $B$, its the unitization $\widetilde B$ is an augmented algebra with augmentation ideal $B$ and, conversely, every augmented algebra is isomorphic to the unitization of its augmentation ideal. Therefore, we will always denote an augmented algebra as $\widetilde A$ where $A$ denotes its augmentation ideal.
 We say that $\widetilde A$ is an \emph{augmented $n$-faced algebra} if it is the unitization of an $n$-faced algebra $A$. In this case $\widetilde A= (A^{(1)}\sqcup\cdots\sqcup A^{(n)})^{\sim}=\widetilde{A}^{(1)}\sqcup_1\cdots \sqcup_1 \widetilde{A}^{(n)}$, where $\sqcup_1$ denotes the free product of unital algebras. If $A,B$ are $*$-algebras, then we consider $\widetilde A$ and $A\sqcup B$ as $*$-algebras in the obvious way.

 A \emph{noncommutative probability space} is a pair $(A,\Phi)$, where $A$ is a unital $*$-algebra and $\Phi$ is a state on $A$. Let $B$ be an $n$-faced $*$-algebra and $(A,\Phi)$ a noncommutative probability space. A $*$-homomorphism $j\colon B\to A$ is called \emph{$n$-faced random variable}. We call $j$ augmented if $B$ is an augmented algebra and $j$ is unital, that is if $B=\widetilde{B'}$ and $j=\widetilde{j\restriction B'}$. Just as for algebras, we will write augmented random variables as $\widetilde \jmath$ from the start with $j$ its restriction to the augmentation ideal. An element of $(b_1,\ldots, b_n)\in A^{n}$ gives rise to the $*$-algebra homomorphism $j_{b_1,\ldots, b_n}\colon\operatorname{*-alg}(b_1)\sqcup\cdots\sqcup \operatorname{*-alg}(b_n)\to A$ determined by $j_{b_1,\ldots, b_n}(b_i)=b_i$. In this sense, we will consider elements of $A^{n}$ as $n$-faced random variables. For an $n$ faced random variable $b=(b_1,\ldots b_n)\in A^n$ and an $m$ tuple $\delta\in [n]^{[m]}$ we define $b^{\delta}:=b_{\delta_1}\cdots b_{\delta_m}$. The numbers $\Phi(b^{\delta})$ are called \emph{moments} of $b$ and the collection of all moments is called \emph{distribution} of $b$. 

In this paper we will only consider two-faced algebras. We write a two-faced algebra as $A=A^{\el}\sqcup A^{\er}$. $A^{\el}$ is called \emph{left face} and $A^{\er}$ is called \emph{right face} of $A$.

A \emph{pointed representation} of an algebra $A$ consists of a pre-Hilbert space $H$, an algebra homomorphism $\pi\colon A\to L(H)$ and a unit vector $\Omega\in H$. For every pointed representation $\pi$ we can define a linear functional $\varphi_\pi(a):=\langle\Omega, \pi(a)\Omega\rangle$. If $\pi$ is a unital $*$-representation, then $\varphi_\pi$ is a state. On the other hand, for any state $\Phi$ on $A$, the GNS-representation yields a  pointed unital $*$-representation. Given a pointed representation on $H$ with unit vector $\Omega$, we denote by $P$ the projection onto $\mathbb C\Omega$ and by $\mathrm{id}$ the identity operator on $H$.

\section{Bimonotone independence}
\label{sec:bimon-indep}

We define the bimonotone product first for pointed representations of two-faced algebras and afterwards for states on augmented two-faced $*$-algebras. 

 \begin{definition}
   Let $A_1, A_2$ be two-faced algebras and $\pi_1,\pi_2$ pointed representations of $A_1,A_2$ on pre-Hilbert spaces $H_1,H_2$ respectively.
   Then we define the pointed representation $\pi_1\bimon\pi_2$ of $A_1\sqcup A_2$ on $H_1\otimes H_2$ with unit vector $\Omega:=\Omega_1\otimes\Omega_2$ by
   \begin{align*}
     \pi_1\bimon\pi_2(a):=
     \begin{cases}
       \pi_1(a)\otimes \mathrm{id}& a\in A_1^{\el}\\
       \pi_1(a)\otimes P & a\in A_1^{\er}\\
       \mathrm{id}\otimes \pi_2(a) & a\in A_2^{\er}\\
       P\otimes\pi_2(a) & a\in A_2^{{\el}}
     \end{cases}
   \end{align*}
   Now suppose that $\varphi_1,\varphi_2$ are states on augmented two-faced $*$-algebras $\widetilde A_1,\widetilde A_2$ and $\pi_i$ are pointed $*$-representations of $A_i$ such that $\varphi_i(a)=\langle\Omega_i,\pi_i(a)\Omega_i\rangle$ for all $a\in A_i$. Then we put
   \[\varphi_1\bimon \varphi_2(a):=\langle\Omega,\pi_1\bimon\pi_2(a) \Omega\rangle\]
   for all $a\in A_1\sqcup A_2$ and $\varphi(1)=1$. 
 \end{definition}

 Note that the definition is independent of the choice of the pointed representations, so we can always take $\widetilde \pi_i$ to be the GNS-represenation of $\varphi_i$. By construction, $\varphi_1\bimon \varphi_2$ is a state, because $\pi:=\pi_1\bimon\pi_2$ is a pointed $*$-representation with $\varphi_1\bimon \varphi_2(a)=\langle\Omega, \widetilde{\pi}(a)\Omega\rangle$ for all $a\in \widetilde{A_1\sqcup A_2}$. But in general $(\pi_1\bimon \pi_2)^{\sim}$ will not be the GNS-represenation of $\varphi_1\bimon \varphi_2$, because $\Omega$ is not necessarily cyclic. 

 It is easy to check that the bimonotone product is associative. Indeed
 \[\pi_1\bimon\cdots\bimon\pi_n(a)=
   \begin{cases}
     P^{\otimes k-1}\otimes \pi_k(a)\otimes \mathrm{id}^{\otimes n-k} & a \in A_k^{\el}\\
     \mathrm{id}^{\otimes k-1}\otimes \pi_k(a)\otimes P^{\otimes n-k} & a \in A_k^{\er}
   \end{cases}
 \]
 holds independently of the way one insert parentheses. Associativity on the level of pointed representations clearly implies associativity on the level of states.

 \begin{remark}
   The bi-monotone product can also be defined for arbitrary linear functionals on algebras. Instead of pointed representations on pre-Hilbert spaces, one can use pointed representations on vector spaces $\widetilde V$ with a fixed decomposition $\widetilde V=\mathbb{C}\Omega\oplus V$. Then every linear functional admits a pointed representation such that $\Phi(a)\Omega=P\pi(a)\Omega$, where $P$ is the projection onto $\mathbb C\Omega$. If $\widetilde V_i=\mathbb{C}\Omega_i\oplus V_i$, then $\widetilde V_1\otimes \widetilde V_2=\mathbb{C}(\Omega_1\otimes\Omega_2)\oplus (\Omega_1\otimes V_2)\oplus (V_1\otimes \Omega_2)\oplus (V_1\otimes V_2)$ yields a direct sum decomposition of $\widetilde V_1\otimes \widetilde V_2$. Given linear functionals $\varphi_i$ on algebras $A_i$, we can find pointed representations $\pi_1,\pi_2$ as above and define
   \[\varphi_1\mathbin{\check\bimon} \varphi_2(a):=\langle\Omega,\pi_1\bimon\pi_2(a) \Omega\rangle\]
   for all $a\in A_1\sqcup A_2$.
 \end{remark}
 
 \begin{theorem}
   The bi-monotone product of linear functionals is a positive $(1,2)$-u.a.u. product in the sense of \cite{MaSc16}, i.e. for all linear functionals $\varphi_i$ on two-faced algebras $A_i$ and all two-faced algebra homomorphisms $j_i\colon B_i\to A_i$ it holds that 
   \begin{itemize}
   \item $(\varphi_1\mathbin{\check\bimon}\varphi_2)\circ \iota_1=\varphi_1$, $(\varphi_1\mathbin{\check\bimon}\varphi_2)\circ \iota_2=\varphi_2$
   \item $(\varphi_1\mathbin{\check\bimon}\varphi_2)\bimon\varphi_3=\varphi_1\bimon(\varphi_2\bimon\varphi_3)$
   \item $(\varphi_1\mathbin{\check\bimon}\varphi_2)\circ (j_1\mathbin{\underline{\sqcup}} j_2)=(\varphi_1\circ j_1)\mathbin{\check\bimon}(\varphi_2\circ j_2)$
   \item $\widetilde{\varphi_1\mathbin{\check\bimon}\varphi_2}$ is a state on $\widetilde A_1\sqcup_1\widetilde A_2$ whenever $\widetilde\varphi_1,\widetilde\varphi_2$ are states on augmented $*$-algebras $\widetilde A_1,\widetilde A_2$ respectively.
   \end{itemize}
   where $\iota_i\colon A_i\hookrightarrow A_1\sqcup A_2$ is the canonical embedding and $j_1\mathbin{\underline{\sqcup}} j_2:=(\iota_1\circ j_1)\sqcup(\iota_2\circ j_2)$.
 \end{theorem}

 \begin{proof}
   Straightforward. Note that $\widetilde{\varphi_1\mathbin{\check\bimon}\varphi_2}=\widetilde\varphi_1\bimon\widetilde\varphi_2$.
 \end{proof}

 Since we are in the realm of universal products, there are associated notions of independence and L{\'e}vy processes for the bi-monotone product, cf.\ \cite{MaSc16} and \cite{GLS16}. 

 \begin{definition}
   Let $j_1,\ldots, j_n$ be two-faced random variables. We call $j_1,\ldots, j_n$ \emph{bi-monotonely independent} if 
   \begin{align}
   \Phi\circ\widetilde {\bigsqcup_i j_i}=(\Phi\circ \widetilde \jmath_1)\bimon\cdots\bimon(\Phi\circ \widetilde\jmath_n).\label{eq:bimon-ind}
   \end{align}
   Pairs of $*$-subalgebras are called bi-monotonely independent if their embedding homomorphisms are.  Pairs of elements are called bi-monotonely independent if their generated $*$-algebras are. 
 \end{definition}

 \begin{lemma}\label{lem:bimon-ind-isom}
   Let $j_i\colon B_i\to A$ be two-faced random variables over a noncommutative probability space $(A,\Phi)$ with distributions $\varphi_i:=\Phi\circ\widetilde\jmath_i$. We denote by $\pi_i,\Pi$ pointed representations on pre-Hilbert spaces $H_i,K$ with unit vectors $\Omega_i,\Omega$ respectively such that
   \[\Phi(a)=\langle\Omega, \Pi(a)\Omega\rangle,\quad \varphi_i(b)=\langle\Omega_i,\pi_i(b)\Omega_i\rangle\]
   for all $a\in A$, $b\in B_i$, $i\in\{1,\ldots,n\}$. Assume that there exists an isometry
   $i\colon H_1\otimes\cdots\otimes H_n\to K$ such that
   \begin{itemize}
   \item $i(\Omega_1\otimes\cdots\otimes\Omega_n)=\Omega$
   \item $i\circ\bigl(\BIGOP{\bimon}_i  \pi_i (b)\bigr)=\Pi(j_k(b))\circ i$ for all $b\in B_k$, $k\in\{1,\ldots, n\}$
   \end{itemize}
   Then $j_1,\ldots, j_n$ are bi-monotonely independent.
 \end{lemma}

 \begin{proof}
   Since both sides of \eqref{eq:bimon-ind} are states, it is enough to prove equality for elements $a_1\cdots a_m\in \bigsqcup_i B_i$. Suppose $a_i\in B_{\varepsilon_i}$ for $i\in 1,\ldots, n$ with $\varepsilon_1,\ldots, \varepsilon_n\in[2]$, $\varepsilon_i\neq\varepsilon_{i+1}$ for $i=1,\ldots, n-1$. In this case we have
   \begin{align*}
     &\Phi\circ (j_1\sqcup\cdots\sqcup j_n) (a_1\cdots a_m)\\
     &= \langle \Omega, \Pi\Bigl(\bigsqcup_i j_i(a_1\cdots a_m)\Bigr)\Omega\rangle\\
     &= \langle i\Omega_1\otimes\cdots\otimes\Omega_n, \Pi(j_{\varepsilon_1}(a_1))\cdots\Pi(j_{\varepsilon_m}(a_m)) i\Omega_1\otimes\cdots\otimes\Omega_n\rangle\\
     &= \langle \Omega_1\otimes\cdots\otimes\Omega_n, i^*i\Bigl(\BIGOP{\bimon}_i \pi_i(a_1)\Bigr) \cdots \Bigl(\BIGOP{\bimon}_i \pi_i(a_m)\Bigr)   \Omega_1\otimes\cdots\otimes\Omega_n\rangle\\
     &= \langle \Omega_1\otimes\cdots\otimes\Omega_n, \Bigl(\BIGOP{\bimon}_i \pi_i(a_1\cdots a_m)\Bigr) \Omega_1\otimes\cdots\otimes\Omega_n\rangle\\
     &= \varphi_1\bimon\cdots \bimon \varphi_n (a_1\cdots a_m)\qedhere
   \end{align*}
 \end{proof}


 
\section{Bimonotone Brownian motion}
\label{sec:bimon-brown-moti}

For a set $M$, put \(M^*:=\bigcup_{n\in\mathbb{N}_{0}} M^n\). In particular, \(\mathbb R^*:=\bigcup_{n\in\mathbb{N}_{0}} \mathbb{R}^n\). We view $\mathbb R^*$ as a measure space with the unique measure whose restriction to $\mathbb R^n$ is $n$-dimensional Lebesgue measure. Restriction of this measure to the subset $\Delta:=\{(t_1,\ldots,t_n)\in\mathbb R^*\mid t_1<\cdots<t_n\}$ makes $\Delta$ a measure space. We will also use the subspaces $\Delta_{[s,t]}:=\Delta \cap [s,t]^{*}$ for the interval $[s,t]$, $s<t$.

\begin{definition}
  Put $\Gamma:=L^2(\Delta)$ and denote by $\Omega$ the function with value $1$ on the empty tuple $\Lambda:=()\in\mathbb{R}^0$ and which vanishes on all $(t_1,\ldots,t_n)$ with $n>0$. We call $\Gamma$ \emph{monotone Fock space}. The unit vector $\Omega$ induces the vector state $\varphi$ with $\varphi(a)=\langle\Omega, a\Omega\rangle$ on $\mathcal A:=\mathcal B(\Gamma)$ called the \emph{vacuum state}. We also define the monotone Fockspaces over intervals as $\Gamma_{[s,t]}:=L^2(\Delta_{[s,t]})$ and view them as subspaces of $\Gamma$ with respect to the obvious embedding. 
\end{definition}

We define the following bounded operators on $\Gamma$:

\begin{definition}
  For every $f\in L^2(\mathbb R)$, we define the \emph{left creation operator} $\lambda^*(f)$, \emph{left annihilation operator} $\lambda(f)$, \emph{right creation operator} $\rho^*(f)$, and  \emph{right annihilation operator} $\rho(f)$ by
  \begin{align*}
    \lambda^*(f)(g)(t_1,\ldots, t_n)&:=f(t_1)g(t_2,\ldots, t_n), \quad(n>0);&\lambda^*(f)(g)(\Lambda)&:=0\\
    \lambda(f)(g)(t_1,\ldots, t_n)&:=\int_{-\infty}^{t_1} \overline{f(\tau)}g(\tau,t_1,\ldots, t_n)\mathrm{d}\tau&&\\
    \rho^*(f)(g)(t_1,\ldots, t_n)&:=g(t_1,\ldots, t_{n-1})f(t_n), \quad(n>0);&\rho^*(f)(g)(\Lambda)&:=0 \\
    \rho(f)(g)(t_1,\ldots, t_n)&:=\int_{t_n}^{\infty} \overline{f(\tau)}g(t_1,\ldots, t_n,\tau)\mathrm{d}\tau&&
  \end{align*}
  respectively. 
\end{definition}

All these operators belong to $\mathcal B(\Gamma)$. It is easy to check that $\lambda^*(f)=(\lambda(f))^*$ and $\rho^*(f)=(\rho(f))^*$. We abbreviate $\lambda^*(1_{[s,t]})$ as $\lambda^*_{s,t}$ and similarly for the other three families of operators. For every interval $[s,t]$ let $B_{[s,t]}$ denote the two-faced $*$-algebra with $B_{s,t}^{\el}=\operatorname{*-alg}(\lambda^*_{s,t})$ and $B_{s,t}^{\er}=\operatorname{*-alg}(\rho^*_{s,t})$

\begin{theorem}
  The two-faced $*$-subalgebras $B_{t_1,t_2}$, $B_{t_2,t_3}$, \ldots, $B_{t_{n-1},t_n}$ are bi-monotonely independent for all choices of $t_1\leq \ldots \leq t_n$.
\end{theorem}

\begin{proof}
  There is a natural isometry $i\colon \Gamma_{[t_1,t_2]}\otimes\cdots\otimes \Gamma_{[t_{n-1},t_n]}\to \Gamma$ with
  \[i(g_1\otimes \cdots\otimes g_{n-1})(\mathbf s)=
    \begin{cases}
      g_1(\mathbf s\cap[t_1,t_2])\cdots g_{n-1}(\mathbf s\cap[t_{n-1},t_n]), &\mathbf s \in \Delta_{[t_1,t_n]}\\
      0, & \text{else}.
    \end{cases}
  \]
  Obviously, $i(\Omega_{[t_1,t_2]}\otimes\cdots\otimes \Omega_{[t_{n-1},t_n]})=\Omega$. Since $\Gamma_{[t_i,t_i+1]}$ is invariant under $B_i$, we can define a pointed representation $\pi_i\colon B_i\to \mathcal {B}(\Gamma_{[t_i,t_i+1]})$ by restriction.  We shortly write $\lambda/\rho^{(*)}_i$ for $\pi_i(\lambda/\rho^{(*)}_{t_i,t_{i+1}})$. So, by Lemma~\ref{lem:bimon-ind-isom} and the definition of the bi-monotone product of pointed representations, we are done if we can show
  \begin{align*}
    \lambda^*_{t_i,t_i+1}\circ i&=i\circ P^{\otimes i-1}\otimes \lambda_i^* \otimes \mathrm{id}^{n-1-i}\\
    \lambda_{t_i,t_i+1}\circ i&=i\circ P^{\otimes i-1}\otimes \lambda_i \otimes \mathrm{id}^{n-1-i}\\
    \rho^*_{t_i,t_i+1}\circ i&=i\circ \mathrm{id}^{\otimes i-1}\otimes \rho_i^* \otimes P^{n-1-i}\\
    \rho_{t_i,t_i+1}\circ i&=i\circ \mathrm{id}^{\otimes i-1}\otimes \rho_i \otimes P^{n-1-i}.
  \end{align*}
  We check the first equality and leave the remaining ones to the reader, as the calculations are very similar.
  Suppose that $s_1<\ldots <s_m$ and that $s_1,\ldots, s_k$ belong to $[t_1,t_{i}]$, $s_{k+1},\ldots,s_{r}$ belong to $[t_{i},t_{i+1}]$ and $s_{r+1},\ldots, s_m$ belong to $[t_{i+1},t_n]$. For all $g_k\in \Gamma_{[t_k,t_{k+1]}}$, $k\in\{1,\ldots, n-1\}$, we get, using the increasing order of $t_1,\ldots, t_n$ and $s_1,\ldots, s_m$,
  \begin{align*}
    &\lambda^*_{t_i,t_i+1}(i(g_1\otimes\cdots\otimes g_{n-1}))(s_1,\ldots, s_m)\\
    &=1_{[t_i,t_{i+t}]}(s_1)\cdot(g_1\otimes\cdots\otimes g_{i-1})(s_2,\ldots, s_k)\cdot \\
    & \hspace*{12em} g_{i}(s_{k+1},\ldots, s_r)\cdot(g_{i+1}\otimes\cdots\otimes g_{n-1})(s_{r+1},\ldots, s_m)\\
    &=
      \begin{cases}
        0 & k>0\\
        g_1(\Lambda)\cdots g_{i-1}(\Lambda) g_{i}(s_2,\ldots, s_{r}) (g_{i+1}\otimes\cdots g_{n-1})(s_{r+1},\ldots, s_m) & k=0, r>0\\
        0 & k=0, r=0
      \end{cases}\\
    &= \bigl(P(g_{1})\otimes\cdots\otimes P(g_{i-1})\otimes \lambda^*_{i}(g_{i})\otimes g_{i+1}\otimes\cdots\otimes g_{n-1}\bigr) (s_1,\ldots,s_m)\\
    &=i\circ (P^{\otimes i-1}\otimes \lambda_i^* \otimes \mathrm{id}^{n-1-i}) (g_1\otimes\cdots\otimes g_{n-1})(s_1,\ldots, s_m).\qedhere
  \end{align*}
\end{proof}

Put $b_{s,t}^{\el}:=\lambda^*_{s,t}+\lambda_{s,t}$ and $b_{s,t}^{\er}:=\rho^*_{s,t}+\rho_{s,t}$.

\begin{theorem}
  The $b_{s,t}=(b_{s,t}^{\el},b_{s,t}^{\er})$, $s,t\geq 0$ form an additive L{\'e}vy process, i.e.
  \begin{itemize}
  \item $b_{r,s}+b_{s,t}=b_{r,t}$ for all $r\leq s\leq t$ and $b_{0,0}=\mathrm{id}$ (increment property)
  \item $\Phi(b_{s,t}^{\delta})=\Phi(b_{0,t-s}^{\delta})$ for all $s\leq t$ and all $\delta\in\{\el,\er\}^*$ (stationarity of increments)
  \item $b_{t_1,t_2}$, $b_{t_2,t_3}$, \ldots, $b_{t_{n-1},t_n}$ are bi-monotonely independent for all $0\leq t_1\leq \ldots \leq t_n$ (independence of increments).
  \end{itemize}
\end{theorem}

\begin{proof}
  The increment property is obvious, independence follows from the previous theorem. Stationarity follows from the fact that the canonical unitary $U_s\colon\Gamma_{s,t}\to\Gamma_{0,t-s}$ with $U_s(f)(t_1,\ldots, t_n)=f(t_1+s,\ldots, t_n+s)$ fulfills $U_s(\Omega)=\Omega$ and $b_{s,t}= U_s^* b_{0,t-s}U_s$. 
\end{proof}

\section{Bi-monotone partitions}
\label{sec:bi-monot-part}

\begin{definition}
  A \emph{partition} of a set $X$ is a set $\pi$ of subsets of $X$, called \emph{blocks} of $\pi$, such that
  \begin{itemize}
  \item $V\neq\emptyset$ for all blocks $V\in\pi$
  \item $V_1\cap V_2=\emptyset$ for all all distinct blocks $V_1,V_2\in\pi$
  \item $\displaystyle \bigcup_{V\in\pi} V=X$
  \end{itemize}
  A partition $\pi$ of $X$ is called a \emph{pair partition} if $\#V=2$ for all blocks $V\in\pi$.

  An \emph{ordered partition} of a set $X$ is a partition $\pi$ of $X$ with a total order $\leq$  between the blocks (i.e.\ a total order on the set $\pi$).
\end{definition}

The special kinds of partitions and ordered partitions we shall define now are important for the study of noncommutative independences, such as freeness (noncrossing partitions), Boolean (interval partitions) and monotone independence (monotona partitions). They appear in moment-cumulant formulas, in the universal coefficients of the correspondung universal product and in the central limit theorems.

\begin{definition}
  Let $X$ be a totally ordered set.

  A partition $\pi$ of $X$ is called
  \begin{itemize}
  \item \emph{noncrossing} if $a,c\in V$ and $b,c\in W$ for elements
    $a<b<c<d$ of $X$ and blocks $V,W\in \pi$ implies $V=W$,
  \item \emph{interval parition}  if $a,c\in V$, $b\in W$ for elements $a<b<c$ of $X$ and blocks $V,W\in \pi$ implies $V=W$.  
  \end{itemize}

  An ordered partition $\pi$ of $X$ is called
  \begin{itemize}
  \item \emph{monotone partition} if $a,c\in V$, $b\in W$ for
    elements $a<b<c$ of $X$ and blocks $V,W\in \pi$ implies $V\geq W$.
  \end{itemize}
\end{definition}

We denote the set of all partitions, noncrossing partitions, interval partitions and monotone partitions of an ordered set $X$ by $\mathrm{P}_{\tensor}(X)$, $\mathrm{P}_{\free}(X)$, $\mathrm{P}_{\bool}(X)$, and $\mathrm{P}_{\mon}(X)$ respectively, according to their associated universal product. An (ordered) partitions is called a \emph{pair partition} if all its blocks have cardinality 2. The set of all pair partitions of type $\odot$ is denoted by $\mathrm{PP}_\odot(X)$ for $\odot\in\{\tensor, \free,\bool,\mon\}$.

Partitions and ordered partitions can be nicely depicted in the following way:
\begin{itemize}
\item The elements of $X$ are written on the bottom line of the diagram, in ascending order if $X$ is ordered
\item Each block is represented by a horizontal line with vertical ``legs'' connecting it down to the points of the block
\item If $\pi$ is not ordered, the heights are chosen such that there are no overlaps and as little crossings as possible.
\item If $\pi$ is ordered, the heights are chosen such that greater blocks have higher horizontal lines. 
\end{itemize}


Noncrossing and monotone partitions are simply recognized as the ones that have no crossing lines in the corresponding diagram.

\begin{definition}
  Let $X$ be a set.
  A \emph{two-faced partition} of $X$ is a partition $\pi$ of $X$ together with a map $\delta\colon X \to\{\up,\down\}$. Similarly, an \emph{ordered two-faced partition} of $X$ is an ordered partition $\pi$ of $X$   together with a map $\delta\colon X \to\{\up,\down\}$.

  The map $\delta$ is called \emph{pattern} of the (ordered) two-faced partition. Frequently, we will consider $\delta$ as an element of $\{\up,\down\}^X$ and write $\delta_x$ instead of $\delta(x)$.
\end{definition}

An (ordered) two-faced partition can be drawn like a partition, but the points of $X$ are put on the bottom and on the top line and the vertical leg at $x\in X$ is drawn down to the bottom line if $\delta(x)=\down$ and up to the top line if $\delta(x)=\up$.

\begin{definition}
  Let $X$ be an ordered set.

  A a two-faced partition of $X$ is called
  \begin{itemize}
  \item \emph{bi-noncrossing partition} if $a,c\in V$ and $b,c\in W$ for elements
    $a<b<c<d$ with $\delta(b)=\delta(c)$ of $X$ and blocks $V,W\in \pi$ implies $V=W$. (cf.\ \cite{CNS15})
  \end{itemize}
  
  An ordered two-faced partition $\pi$ of $X$ is called
  \begin{itemize}
  \item \emph{bi-monotone partition} if $a,c\in V$, $b\in W$ for
    elements $a<b<c$ of $X$ and blocks $V,W\in \pi$ implies $V\geq W$ if $\delta(b)=\down$ and $V\leq W$ if $\delta(b)=\up$. 
  \end{itemize}
\end{definition}

Again, the bi-noncrossing and bi-monotone partitions exactly correpond to those diagrams which have no crossing lines. The sets of bi-noncrossing partitions, bi-monotone partitions, bi-noncrossing pair partitions, and bi-monotone pair partitions of $X$ with pattern $\delta\colon X\to\{\el,\er\}$  are denoted by $\mathrm{P}_{\bifree}(\delta)$, $\mathrm{P}_{\bimon}(\delta)$,  $\mathrm{PP}_{\bifree}(\delta)$, and $\mathrm{PP}_{\bimon}(\delta)$ respectively. 

\begin{example}
  The two-faced partition $(\pi,\delta)$ of $\{1,\dots, 6\}$ with
  \[\pi=\Bigl\{\{1,4\},\{2,3\},\{5,6\}\Bigr\},\quad \delta=(\down,\down,\down,\up,\up,\up)\]
  is bi-noncrossing, as it can  be drawn in a noncrossing way, see Figure~\ref{bncpp-ex}.
  \begin{figure}\caption{bi-noncrossing partition $\bigl\{\{1,4\},\{2,3\},\{5,6\}\bigr\}$ with pattern $\delta=(\down,\down,\down,\up,\up,\up)$}\label{bncpp-ex}
    \includegraphics{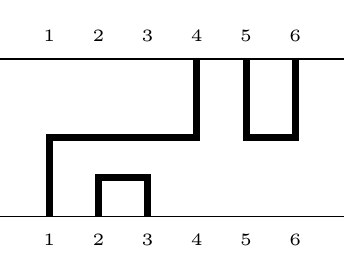}
  \end{figure}
  There are three possible orders between the blocks such that the resulting ordered two-faced partition is bi-monotone. Indeed, the block $\{1,4\}$ has to be smaller than the block $\{2,3\}$, whereas $\{5,6\}$ can be in any position. This is easily seen in the corresponding diagrams in Figure~\ref{bmord-ex}.
  \begin{figure}\caption{bi-monotone orderings of a bi-noncrossing partition}\label{bmord-ex}
    \includegraphics{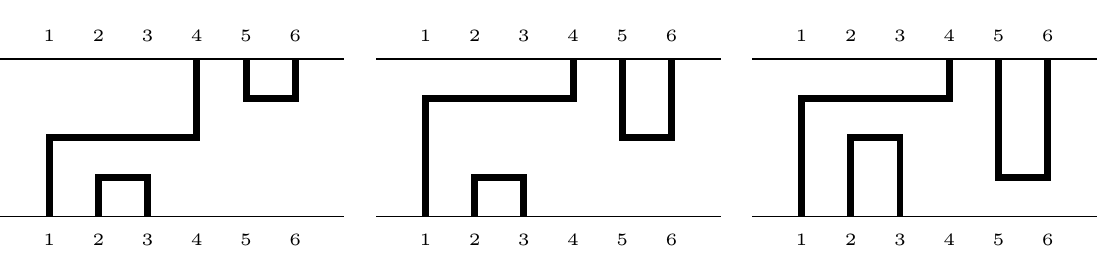}
  \end{figure}  
\end{example}

\begin{remark}
  Let $X$ be an ordered $2n$-set. Then one can show that for every pattern $\delta\colon X\to\{\el,\er\}$, there are at most $(2n-1)!!$ bi-monotone pair partitions with pattern $\delta$. A complete proof requires quite a bit of combinatorial background which is not needed for the central limit theorem, so we will present it in a second paper. Here we only sketch the key idea.
  
A partition $\pi$ of an ordered set $X$ is called \emph{irrecucible} if for all $a\neq\max X$ there is a block $V\in\pi$ with elements $b,c\in V$ such that $b\leq a< c$. For a set $M$ of partitions of $X$ we denote by $\mathrm{I}M$ the set of irreducible partitions in $M$. 
For monotone partitions (bi-monotone partitions with constant pattern), we know that $\#\mathrm{PP}(X)=(2n-1)!!$. Also note that the highest block of an irreducible monotone partition has to be $\{\min X,\max X\}$, so deleting the highest block yields a bijection between $\mathrm{IPP}_{\mon}(X)$ and $\mathrm{PP}_{\mon}('X')$, where $'X':=X\setminus\{\min X,\max X\}$.

  For a general pattern $\delta\colon X\to \{\el,\er\}$, it is not difficult to see that
  \[ \#\mathrm{PP}_{\bimon}(\delta)=\sum_{\substack{\delta=\delta_1\smile\cdots \smile\delta_k\\|\delta_i|\in 2\mathbb N}} \#\mathrm{IPP}_{\bimon}(\delta_1)\cdots\#\mathrm{IPP}_{\bimon}(\delta_k) \binom{|\delta|/2}{|\delta_1|/2,\ldots,|\delta_k|/2}. \]
  On the right hand side the induction hypothesis can be applied to replace all partitions with shorter patterns by ordinary monotone partitions. Finally, one has to deal with the the summand  $\mathrm{IPP}_{\bimon}(\delta)$. The Dyck path description of noncrossing pair partitions (see e.g.\ \cite[Exercise 8.23]{NiSp06}) also works for bi-noncrossing pair partitions with fixed pattern, and one can use it to find an injection $\mathrm{IPP}_{\bimon}(\delta)\hookrightarrow\mathrm{PP}_{\bimon}('\delta')$, where $'\delta':=\delta\restriction{'X'}$. 
  Then the induction hypothesis yields
  \begin{align*}
    \#\mathrm{IPP}_{\bimon}(\delta)
    \leq \#\mathrm{PP}_{\bimon}('\delta')\leq\#\mathrm{PP}_{\mon}('X')=\#\mathrm{IPP}_{\mon}(X).
  \end{align*}
  \end{remark}

\section{Bi-monotone central limit theorem}
\label{sec:bi-monotone-central}

Our proof of the bi-monotone central limit theorem is based on the central limit theorem for singleton independence as given in {\cite[Lemma 2.4]{AHO98}}. For convenience of the reader, we recall the relevant part of the theorem.
\begin{theorem}[cf.\ {\cite[Lemma 2.4]{AHO98}}]\label{theo:SCLT}
  Let $J$ be a set and $(A,\Phi)$ a noncommutative probability space. Assume that $b^{(j)}=(b^{(j)}_n)_{n=1}^{\infty}$, $j\in J$, sequences of elements in $A$ such that
  \begin{itemize}
  \item each $b_n^{(j)}$ is centered, i.e. $\Phi(b_n^{(j)})=0$, 
  \item the condition of boundedness of mixed moments is fulfilled, i.e.  for each $k\in\mathbb N$ there exists a positive constant $\nu_k$ such that $\left\vert\Phi(b_{n_1}^{(j_1)}\cdots b_{n_k}^{(j_k)})\right\vert\leq\nu_k$ for any choice of $n_1,\ldots, n_k\in\mathbb N$ and $j_1,\ldots, j_k\in J$,
  \item the singleton condition is satified, i.e.  for any choice of $k\in\mathbb N$, $j_1,\ldots, j_k\in J$ and $n_1,\ldots, n_k\in\mathbb N$
    \[\Phi(b_{n_1}^{(j_1)}\cdots b_{n_k}^{(j_k)})=0\]
holds whenever there exists an index $n_s$ which is different from all other ones.
  \end{itemize}
  Write $S_N(b^{(j)}):=\sum_{n=1}^{N}b^{(j)}_n$. Then
  \begin{multline}
    \lim_{N\to \infty} \Phi\left(\frac{S_N(b^{(j_1)})}{N^{1/2}}\cdots \frac{S_N(b^{(j_{2n})})}{N^{1/2}}\right)\\
      =\lim_{N\to\infty} N^{-n} \sum_{\substack{\pi\colon[2n]\to[n]\\\text{2 to 1}}}\sum_{\substack{\sigma\colon[n]\to[N]\\\text{order preserving}}}\Phi\left(b^{(j_1)}_{\sigma\circ\pi(1)}\cdots b^{(j_{2n})}_{\sigma\circ\pi(2n)}\right).\label{eq:2}
  \end{multline}
\end{theorem}

We will use it in the form of the following Corollary.

\begin{corollary}\label{cor-SCLT}
  Let $b^{(j)}=(b^{(j)}_n)_{n=1}^{\infty}$, $j\in \{\el,\er\}$, be sequences of two-faced random variables such that
  \begin{itemize}
  \item each  $b^{(j)}_n$ is centered
  \item the singleton condition is fulfilled
  \item the distribution is \emph{spreadable}, i.e.\
    \[\Phi(b_{n_1}^{j_1}\cdots b_{n_k}^{j_k})=\Phi(b_{\sigma(n_1)}^{j_1}\cdots b_{\sigma(n_k)}^{j_k})\]
    for each order preserving $\sigma\colon[n]\to\mathbb N$.
  \end{itemize}
  Then
  \begin{align}
    \lim_{N\to \infty} \Phi\left(\frac{S_N(b^{(j_1)})}{N^{1/2}}\cdots \frac{S_N(b^{(j_{2n})})}{N^{1/2}}\right)
      =\frac{1}{n!}\sum_{\substack{\pi\colon[2n]\to[n]\\\text{2 to 1}}}\Phi\left(b^{(j_1)}_{\pi(1)}\cdots b^{(j_{2n})}_{\pi(2n)}\right).
  \end{align}
\end{corollary}

\begin{proof}
  Spreadability implies boundedness of mixed moments.
  There are exactly $\binom{N}{n}$ order preserving maps from $[n]$ to $[N]$, and by spreadability the value of $\Phi\left(b^{(j_1)}_{\sigma\circ\pi(1)}\cdots b^{(j_{2n})}_{\sigma\circ\pi(2n)}\right)$ is independent of $\sigma$.
  Finally, note that
  \[\lim_{N\to\infty}  N^{-n}\binom{N}{n}=\frac{1}{n!}\lim_{N\to\infty}\frac{N(N-1)\cdots (N-n+1)}{N^n}=\frac{1}{n!}.\qedhere\]
\end{proof}



\begin{lemma}\label{lem:bi-mon-clt-1}
  Let $(B_1^{\el}, B_1^{\er}),\ldots, (B_n^{\el},B_n^{\er})$ be bi-monotonely independent pairs of subalgebras. Let $b_i\in B_{\varepsilon_i}^{\delta_i}$, $i\in\{1,\ldots, m\}$, with $\varepsilon_i\in\{1,\ldots, n\}$ and $\delta_i\in\{\er,\el\}$. If the two-faced partition $(\pi,\delta)$ with blocks $V_i:=\{k\mid \varepsilon_k=i\}$, $V_1<\ldots<V_n$, and pattern $\delta$ is bimonotone, then
  \[\Phi(b_1\cdots b_m)=\Phi\Bigl(\prod^{\rightarrow}_{\varepsilon_k=1} b_k\Bigr)\cdots \Phi\Bigl(\prod^{\rightarrow}_{\varepsilon_k=n} b_k\Bigr).
\]
\end{lemma}

\begin{proof}
  Let $\pi_i$ be pointed representations with $\Phi(b)=\langle\Omega_i,\pi_i(b)\Omega_i\rangle$ for all $b\in B_i$, $i\in[n]$. On the representation space of each $\pi_i$, we denote by $P$ the projection onto $\mathbb C\Omega_i$ and by $\mathrm{id}$ the identity operator.
  For $b\in B_\varepsilon^{\delta}$ with $\varepsilon\in [n]$, $\delta\in\{\el, r\}$, we write $\BIGOP{\bimon}_i \pi_i(b)$ as $T_1(b)\otimes\cdots \otimes T_n(b)$ with
  \[T_i(b)=
    \begin{cases}
      P&\text{if ($i<\varepsilon$ and $\delta=\el$) or ($i>\varepsilon$ and $\delta=\er$)}\\
      \mathrm{id}&\text{if ($i<\varepsilon$ and $\delta=r$) or ($i>\varepsilon$ and $\delta=\el$)}\\
      \pi_i(b)&\text{if $i=\varepsilon$}.
    \end{cases}
  \]
  With this notation
  \begin{align*}
    \BIGOP{\bimon}_i \pi_i(b_1\cdots b_n)
    &=\Bigl(\BIGOP{\bimon}_i \pi_i(b_1)\Bigr) \cdots \Bigl(\BIGOP{\bimon}_i \pi_i(b_m)\Bigr)\\
    &=\bigl(T_1(b_1)\cdots T_1(b_n)\bigr)\otimes\cdots\otimes \bigl(T_n(b_1)\cdots T_n(b_n)\bigr).
  \end{align*}
  One checks that the fact that $(\pi,\delta)$ is bimonotone implies that $T_i(b_j)=\mathrm{id}$ whenever there are $\mu,\nu$ with $\mu<j<\nu$, $\varepsilon_\mu=\varepsilon_\nu=i$ and $\varepsilon_j\neq i$. Then it follows easily that \[\langle\Omega_i,T_i(b_1)\cdots T_i(b_n)\Omega_i\rangle=\Phi\Bigl(\prod^{\rightarrow}_{\varepsilon_k=i} b_k\Bigr)\] and therefore
  \begin{align*}
    \Phi(b_1\cdots b_m)
    &=\left\langle\Omega_1\otimes\cdots\otimes\Omega_n, \BIGOP{\bimon}_i \pi_i(b_1\cdots b_n)\Omega_1\otimes\cdots\otimes\Omega_n\right\rangle\\
    &=\Bigl\langle\Omega_1,T_1(b_1)\cdots T_1(b_n)\Omega_i\Bigr\rangle\cdots \Bigl\langle\Omega_n,T_n(b_1)\cdots T_n(b_n)\Omega_i\Bigr\rangle\\
    &=\Phi\Bigl(\prod^{\rightarrow}_{\varepsilon_k=1} b_k\Bigr)\cdots \Phi\Bigl(\prod^{\rightarrow}_{\varepsilon_k=n} b_k\Bigr)
  \end{align*}
  as claimed.
\end{proof}

\begin{lemma}\label{lem:bi-mon-clt-2}
  Let $(b_{1}^{(\el)},b_{1}^{(\er)}),\ldots, (b_{n}^{(\el)},b_{n}^{(\er)})\in A\times A$ be bi-monotonely independent and centered. For an ordered two-faced pair partition $(\pi,\delta)$ with blocks $V_i:=\{k\mid \varepsilon_k=i\}$, $V_1<\ldots<V_n$, and pattern $\delta=(\delta_1,\ldots, \delta_{2n})$, it holds that
  \[\Phi(b_\pi)=\Phi(b_{\varepsilon_1}^{(\delta_1)}\cdots b_{\varepsilon_{2n}}^{(\delta_{2n})})=  0\]
  whenever $(\pi,\delta)$ is not bi-monotone.
\end{lemma}

\begin{proof}
  If $(\pi,\delta)$ is not bi-monotone, there are $\mu<j<\nu$ and $i\in[n]$ with $\varepsilon_\mu=\varepsilon_\nu=i$ and either
  \begin{enumerate}
  \item $\delta_j=\up$ and $\varepsilon_j>i$ or
  \item $\delta_j=\down$ and $\varepsilon_j<i$.
  \end{enumerate}
  In either case, using the notation of the previous Lemma,
  \begin{align*}
    T_i(b_{\varepsilon_\mu}^{(\delta_\mu)})&=\pi_i(b_{\varepsilon_\mu}^{(\delta_\mu)})\\
    T_i(b_{\varepsilon_j}^{(\delta_j)})&=P\\
    T_i(b_{\varepsilon_\nu}^{(\delta_\nu)})&=\pi_i(b_{\varepsilon_\nu}^{(\delta_\nu)}),
  \end{align*}
  so
  \[\langle\Omega_i,T_i(b_{\varepsilon_1}^{(\delta_1)})\cdots T_i(b_{\varepsilon_{2n}}^{(\delta_{2n})})\Omega_i\rangle=\Phi(b_{i}^{(\delta_\mu)})\Phi(b_{i}^{(\delta_\nu)})=0\qedhere\]
\end{proof}

\begin{theorem}[Bi-monotone central limit theorem]\label{theo:bi-mon-clt}
    Let  $b^{(j)}=(b^{(j)}_n)_{n=1}^{\infty}$, $j\in \{\el,\er\}$, be sequences of elements in a quantum probability space such that
  \begin{itemize}
  \item each  $b^{(j)}_n$ is centered
  \item the sequence of pairs $(b_n^{(\el)},b_n^{(\er)})_{n\in\mathbb N}$ is bi-monotonely independent
  \item the pairs $(b_n^{(\el)},b_n^{(\er)})$ have the same distribution for all $n\in\mathbb N$.
  \end{itemize}
  Put $c_{p,q}:=\Phi(b_n^{(p)}b_n^{(q)})$ for $p,q\in\{\el,\er\}$. Then for every pattern $\delta=(\delta_1,\ldots, \delta_{2n})$ we have
  \begin{align*}
    \lim_{N\to \infty} \Phi\left(\frac{S_N(b^{(\delta_1)})}{N^{1/2}}\cdots \frac{S_N(b^{(\delta_{2n})})}{N^{1/2}}\right)
      =\frac{1}{n!}\sum_{\pi\in \mathrm{PP}_{\bimon}(\delta)} \prod_{\{k<l\}\in\pi} c_{\delta_k,\delta_l}.
  \end{align*}
\end{theorem}

\begin{proof}
  The singleton condition and spreadibility of the distribution follow from bimonotone independence. Therefore we can apply Corollary~\ref{cor-SCLT} to get
   \begin{align*}
    \lim_{N\to \infty} \Phi\left(\frac{S_N(b^{(\delta_1)})}{N^{1/2}}\cdots \frac{S_N(b^{(\delta_{2n})})}{N^{1/2}}\right)
      =\frac{1}{n!}\sum_{\substack{\pi\colon[2n]\to[n]\\\text{2 to 1}}}\Phi\left(b^{(\delta_1)}_{\pi(1)}\cdots b^{(\delta_{2n})}_{\pi(2n)}\right).
   \end{align*} With every 2-to-1 map $\pi\colon[2n]\to[n]$ we associate the ordered pair partition $\hat\pi$ with blocks $V_i:=\pi^{-1}(i)$, $V_1<\ldots<V_n$ and the two-faced ordered pair partition $(\hat\pi,\delta)$. Thus, we can combine Lemma~\ref{lem:bi-mon-clt-1} and Lemma~\ref{lem:bi-mon-clt-2} to get
   \begin{align*}
    \Phi\left(b^{(\delta_1)}_{\pi(1)}\cdots b^{(\delta_{2n})}_{\pi(2n)}\right)=
     \begin{cases}
       0&\text{if $(\hat\pi,\delta)$ is not bi-monotone}\\
       \displaystyle\prod_{\{k<l\}\in\pi} c_{\delta_k,\delta_l}& \text{if $(\hat\pi,\delta)$ is bi-monotone}
     \end{cases}
   \end{align*}
   and the proof is finished.
\end{proof}


\section{Distribution of bi-monotone Brownian motion}
\label{sec:distr-bi-monot}




Next, we will apply the bi-monotone central limit theorem in order to determine the distribution of bi-monotone Brownian motion.  

\begin{theorem}
  Put $b_{s,t}^{(\el)}:=\lambda^*_{s,t}+\lambda_{s,t}$ and $b_{s,t}^{(\er)}:=\rho^*_{s,t}+\rho_{s,t}$. Then
  \[\langle\Omega,b_{0,1}^{(\delta_1)}\cdots b_{0,1}^{(\delta_n)} \Omega\rangle=\frac{1}{n!}\#\mathrm{PP}_{\bimon}(\delta).\]
\end{theorem}

\begin{proof}
  The proof is based on two simple observations:
  \begin{itemize}
  \item For any $N\in\mathbb N$, we can write $b_{0,1}^{(j)}=b_{0,\frac{1}{N}}^{(j)}+b_{\frac1N,\frac2N}^{(j)}+\cdots+b_{\frac{N-1}{N},1}^{(j)}$.
  \item The two families
    \[\left(b_{\frac{i}{n},\frac{i+1}{n}}^{(j)}\right)_{\substack{i\in\mathbb{N}\\j\in\{\el,\er\}}} \quad\text{and}\quad \left(\frac{b_{i,i+1}^{(j)}}{\sqrt{n}}\right)_{\substack{i\in\mathbb{N}\\j\in\{\el,\er\}}}\]
    have the same distribution, i.e.
    \[\Phi\left(b_{\frac{i_1}{n},\frac{i_1+1}{n}}^{(j_1)}\cdots b_{\frac{i_k}{n},\frac{i_k+1}{n}}^{(j_k)}\right)=N^{-\frac{k}{2}}\Phi\left(b_{i_1,i_1+1}^{j_1}\cdots b_{i_k,i_k+1}^{j_k}\right)\]
  \end{itemize}
  for all $i_1,\ldots, i_k\in\mathbb{N}$, $j_1,\ldots, j_k\in\{\el,\er\}$. Now put $b_i^{(j)}:=b_{i,i+1}^{(j)}$. The sequences $(b_i^{(j)})$ consist of centered bounded operators and are bi-monotonely independent. Therefore, we can apply the bi-monotone central limit theorem. Since $c_{p,q}:=\Phi(b_{i}^{(p)}b_{i}^{(q)})=1$ for all $i\in\mathbb N$ and all $p,q\in\{\el,\er\}$, the stated formula follows.
  %
\end{proof}

\begin{corollary}
  Put $b_{s,t}:=\lambda^*_{s,t}+\lambda_{s,t}+\rho^*_{s,t}+\rho_{s,t}$. Then
  \[\langle\Omega,b_{0,1}^n\Omega\rangle=\frac{1}{n!}\#\mathrm{PP}_{\bimon}(n),\]
  where $\mathrm{PP}_{\bimon}(n)$ is the set of all bi-monotone pair partitions of $[n]$ with arbitrary pattern.
\end{corollary}

\begin{proof}
  Since $b_{s,t}=b_{s,t}^{(\el)}+b_{s,t}^{(\er)}$, this follows directly from the previous theorem and
  \[b_{s,t}^n=(b_{s,t}^{(\el)}+b_{s,t}^{(\er)})^{n}=\sum_{\delta\in\{\el,\er\}^n}b_{s,t}^{(\delta_1)}\cdots b_{s,t}^{(\delta_n)}.\qedhere\]
\end{proof}

This corollary allows us to calculate the number of bi-monotone pair partitions. The number of bi-monotone partitions of a $2n$-set for $n=0,1,\ldots,10$ are:\\
1,
 4,
 48,
 928,
 24448,
 811776,
 32460032,
 1516774912,
 81064953344,
 4876115246080,
 325959895390976.\\
 We do not know any explicit or recursive formula for these numbers.

 Since $b_{0,1}$ is a bounded selfadjoint operator on a Hilbert space, its moments are the moments of a uniquely determined compactly supported probability measure on $\mathbb R$. It would be nice to have an explicit formula of this probability measure.


 \bibliographystyle{amsalpha}
 \linespread{1.25}

 \addcontentsline{toc}{chapter}{References}
\bibliography{../bibtex/bib/mybibs/maltebib}

\setlength{\parindent}{0pt}
\end{document}